\documentclass[12pt]{amsart}
\usepackage[utf8]{inputenc}

\title{Local central limit theorem for triangle counts in sparse random graphs}
\author[P.\ Ara\'ujo]{Pedro Ara\'ujo}
\address{Department of Mathematics, Faculty of Nuclear Sciences and Physical Engineering, Czech Technical University in Prague, Trojanova 13, 120 00, Prague, Czech Republic}
\email{pedro.araujo@cvut.cz}

\author{Let\'icia Mattos}
\address{Institut für Informatik, Universität Heidelberg, Im Neuenheimer Feld 205, 69120 Heidelberg, Germany.}\email{mattos@uni-heidelberg.de}

\thanks{During most of this work, P. Ara\'ujo was supported by the Czech Science Foundation, grant number 20-27757Y, while afilliated with Institute of Computer Science of the Czech Academy of Sciences, with institutional support RVO:67985807. 
L.~Mattos was supported by the Deutsche Forschungsgemeinschaft (DFG, German Research Foundation) under Germany's Excellence Strategy – The Berlin Mathematics Research Center MATH+ (EXC-2046/1, project ID: 390685689).}

\usepackage[utf8]{inputenc}
\usepackage{ mathrsfs }
\usepackage{ dsfont }
\usepackage{amsmath}
\usepackage{indentfirst}
\usepackage{mathtools}
\usepackage{amsfonts}
\usepackage{amsthm}
\usepackage{graphicx}
\usepackage[english]{babel}
\usepackage{fullpage}
\usepackage{enumerate}
\usepackage{tasks}
\usepackage{color}
\usepackage{theoremref}
\usepackage{relsize}

\usepackage{ amssymb }
\usepackage{setspace}
\usepackage{ textcomp }

\newtheorem{theorem}{Theorem}

\newtheorem{lemma}[theorem]{Lemma}

\newtheorem{corollary}[theorem]{Corollary}

\numberwithin{theorem}{section}
\theoremstyle{remark}

\def\cE{\mathbb{E}}

\def\LL{\mathcal{L}}

\def\N{\mathcal{N}}

\def\se{\subseteq}

\newcommand{\eps}{\varepsilon}
\renewcommand{\Pr}[1]{\mathbb{P}\left(#1\right)}
\newcommand{\Ex}[1]{\mathbb{E}\left(#1\right)}
\newcommand{\Var}[1]{\mathbin{{\mathbb{V}}\mkern-2mu{\text{ar}}\left(#1\right)}}
\definecolor{lblue}{rgb}{0.5,0.5,1}

\newcommand{\eq}[1]{\begin{equation}\label{eq:#1}}
	\newcommand{\eqe}{\end{equation}}

\begin{document}
	
	\begin{abstract}
    Let $X_H$ be the number of copies of a fixed graph $H$ in $G(n,p)$.
		In 2016, Gilmer and Kopparty conjectured that a local central limit theorem should hold for $X_H$ as long as $H$ is connected, $p\gg n^{-1/m(H)}$ and $n^2(1-p)\gg 1$, where $m(H)$ denotes the $m$-density of $H$.
		Recently, Sah and Sawhney showed that
		the Gilmer--Kopparty conjecture holds for constant $p$.
		In this paper, we show that the Gilmer--Kopparty conjecture holds for triangle counts in the sparse range.
	More precisely,
        if $p \in (4n^{-1/2}, 1/2)$, then
		\[
			\sup_{x\in \mathcal{L}}\left| \dfrac{1}{\sqrt{2\pi}}e^{-x^2/2}-\sigma\cdot \mathbb{P}(X^* = x)\right|=n^{-1/2+o(1)}p^{1/2},
		\]
		where $\sigma^2 = \mathbb{V}\text{ar}(X_{K_3})$, $X^{*}=(X_{K_3}-\mathbb{E}(X_{K_3}))/\sigma$ and $\mathcal{L}$ is the support of $X^*$.
		By combining our result with the results of Röllin--Ross and Gilmer--Kopparty, this establishes the Gilmer--Kopparty conjecture for triangle counts for $n^{-1}\ll p < c$, for any constant $c\in (0,1)$. Our quantitative result is enough to prove that the triangle counts converge to an associated normal distribution also in the $\ell_1$-distance. This is the first local central limit theorem for subgraph counts above the so-called $m_2$-density threshold.
	\end{abstract}
	\maketitle

	\section{Introduction}

 \setstretch{1.4}{
Let $G(n,p)$ be the Erd\H{o}s--R\'enyi random graph on the vertex set $[n]=\{1,2,\cdots,n\}$, where each edge of $K_n$ is included independently with probability $p$.
	Our interest lies in the random variable $X_H$ that counts the number of copies of a fixed graph $H$ in $G(n,p)$. 
	The study of subgraph counts goes back to the 1960s with the seminal paper of Erd\H{o}s and Rényi~\cite{erd6s1960evolution}, where they showed that the threshold for the appearance of a fixed graph $H$ in $G(n,p)$ is of order $n^{-1/m(H)}$, with $m(H):=\max\{e(J)/v(J) : J\subset H\}$.
	When $p$ is at the threshold,
	Barbour~\cite{barbour1982poisson} and, independently, Karo{\'{n}}ski and Ruci{\'{n}}ski~\cite{kar-ruci} showed that $\Pr{X_H=x}$ is asymptotically equal to the probability that a Poisson random variable is equal to $x$.

	The distribution of $X_H$ past the threshold is also well-understood. 
	To explain that, we write  $\mu_H := \Ex{X_H}$ and $\sigma_H^2 := \Var{X_H}$, and we let $X_H^{*}:=(X_H-\mu_H)/\sigma_H$ be the normalised subgraph count.
	We say that $X_H^{*}$ converges in distribution to a random variable $X$, or write $X_H^{*} \xrightarrow{d} X$ for short, if for every $x \in \mathbb{R}$ we have $\Pr{X_H^{*} < x} \to \Pr{X < x}$ as $n$ tends to infinity.
	Early developments in this subject (see ~\cite{kar-ruci,nowicki1988subgraph}) culminated in the work of
	Ruci{\'{n}}ski~\cite{rucinski1988small}, who showed that $X_H^{*}$ converges in distribution to the standard normal random variable $N(0,1)$ if and only if $p \gg n^{-1/m(H)}$ and $n^2(1-p)\gg 1$.
	Later, Barbour, Karo\'nski and Ruci\'nski~\cite{barbour1989central} were able to substantially improve the bounds on the error function in the central limit theorem for $X_H$.
	In the 1990s, Janson~\cite{janson1994orthogonal} built upon the techniques used in~\cite{nowicki1988subgraph} to study central limit theorems for joint distributions of subgraph counts.
	
	A fairly natural question which arises from the work of Ruci{\'{n}}ski~\cite{rucinski1988small}  is whether $\Pr{X_H = x}$ is asymptotically equal to the density of the normal random variable $N(\mu_{H},\sigma_{H})$ at $x$, for every $x$ in the support of $X_H$.
	A positive answer for this question has been referred to as a local central limit theorem. 
	In 2015,
	R\"ollin and Ross~\cite{rollin2015local}  showed that if $n^{-1} \ll p = O(n^{-1/2})$, then a local central limit theorem holds for triangle counts.
	Later, this result was extended by
	Gilmer and Kopparty~\cite{gilmer2016local} for triangle counts and every constant $p \in (0,1)$.
	They also conjectured that a local central limit theorem should hold for $X_H$ for every connected graph $H$ as long as $p \gg n^{-1/m(H)}$ and $n^2(1-p)\gg 1$.
    Our main theorem below states that indeed the Gilmer--Kopparty conjecture is true for triangle counts by establishing the first local central limit theorem above the so-called $m_2$-density threshold -- specifically at  $n^{-1/2}$ for triangle counts --
    and thus closing the gap in the results mentioned above.

	\begin{theorem}
		\label{thm:main}
		Let $X$ be the number of triangles in $G(n,p)$, with $\mu = \Ex{X}$ and $\sigma^2=\Var{X}$, and let $\eps\in (0,1)$ be any constant. If $p\in (4n^{-1/2}, 1/2)$, then
		\begin{align}\label{eq:main-result}
			\sup_{x\in \LL}\left|\thinspace \frac{1}{\sqrt{2\pi}}e^{-x^2/2}-\sigma\cdot \Pr{X^* = x} \right| = O(n^{-1/2+\eps}p^{1/2}),
		\end{align}
		where $\mathcal{L}$ is the support of $X^{*}$.
	\end{theorem}

	As mentioned before, a local central limit theorem for triangle counts was obtained by R\"ollin and Ross~\cite{rollin2015local} in the range where $n^{-1} \ll p = O(n^{-1/2})$.
	This result was extended by Gilmer and Kopparty~\cite{gilmer2016local} for constant $p \in (0,1)$.
	Our contribution is in the range where $Cn^{-1/2}\le p \le 1/2$, for some constant $C>0$. We believe that the lower bound of $p$ is related, in the general setting, to the so-called $m_2$-density. 	As a consequence of Theorem~\ref{thm:main}, we were also able to obtain a quantitative bound on the $\ell_1$-distance between the normalised triangle count $X^*_{K_3}$ and the associated normal distribution in the regime $p=\Omega(n^{-1/2})$.

	\begin{theorem}
		\label{thm:main-2}
		Let $X$ be number of triangles in $G(n,p)$, with $\mu = \Ex{X}$ and $\sigma^2=\Var{X}$, and let $\eps\in (0,1)$ be any constant.
        If $p \in (4n^{-1/2},1/2)$, then
		\begin{align}\label{eq:main-result-2}
			\sum \limits_{x\in \mathbb{N}}\left|\thinspace \frac{1}{\sqrt{2\pi \sigma^2}}e^{-\frac{(x-\mu)^2}{2\sigma^2}}- \Pr{X = x} \right| = O\left(\dfrac{p^{1/2}}{ n^{1/2-\eps}}\right).
		\end{align}
	\end{theorem}

}
\vspace*{-0.5cm}
\setstretch{1.15}{
     One interesting consequence of a local limit theorem concerns the phenomena called \textit{anti-concentration}. Since concentration inequalities play a crucial r\^ole in probabilistic combinatorics, it is compelling to ask how small of an interval we can expect a random variable to be in. Theorem \ref{thm:main} answers this question by showing that $\Pr{X=x}=O(\sigma^{-1})$ for every $x \in \mathbb{N}$. 
     We emphasize that the implied constant in the $O$-notation is independent of both $x$ and $p$.
     Thus, with high probability $X$ is not contained in any fixed interval of length $o(\sigma)$.
     \vspace*{-0.7cm}
    \begin{corollary}
        Let $X$ be number of triangles in $G(n,p)$, $\sigma^2=\Var{X}$ and $\varepsilon \in(0,1)$ be any constant. If $n^{-1}\ll p\leq 1-\varepsilon$, then with high probability $X$ is not contained in any fixed interval of length $o(\sigma)$.
    \end{corollary}
\vspace*{-0.7cm}
 
	Local central limit theorems for general subgraph counts were established for $p$ constant. Berkowitz~\cite{berkowitz2016quantitative,berkowitz2018local} extended the results of Gilmer and Kopparty~\cite{gilmer2016local} by obtaining quantitative local limit theorems for clique counts when $p$ is constant. 
	More precisely, they showed that the supremum in~\eqref{eq:main-result} not only goes to zero but decays
	polynomially fast for clique counts. Around the same time, Fox, Kwan and Sauermann~\cite{10.1214/20-AOP1490,fox_kwan_sauermann_2021} showed an almost-optimal anti-concentration result on $X_H$ for every connected graph $H$.
	That is, they showed that if $p \in (0,1)$ is constant and $H$ is a connected graph, then $\max_{x \in \mathbb{Z}} \Pr{X_H = x} \le \sigma_H^{-1+o(1)}$.
	This result was strengthened by Sah and Sawhney~\cite{sah2022local}, who obtained a quantitative local central limit theorem for $X_H$ for every connected graph $H$ and constant $p \in (0,1)$. 
    Fortunately for the authors of this paper, their techniques seem to break down at some point as $p$ tends to $0$ and new ideas were necessary to establish a local central limit theorem in the sparse range.

	In the proof of Theorem \ref{thm:main}, as in previous approaches~\cite{berkowitz2016quantitative,berkowitz2018local,gilmer2016local,sah2022local}, we use the Fourier inversion formula to transform our problem into estimating the difference of characteristic functions of $X^*$ and of the normal distribution, i.e. $\big|\Ex{e^{itX^*/\sigma}}-e^{-t^2/2}\big|$ for $t\in[-\pi \sigma, \pi \sigma]$.
    For $t$ roughly in the range $ [-(p^2n)^{1/2},(p^2n)^{1/2}]$, 
	we apply the Stein's Method following the steps of Barbour, Karo\'nski and Ruci\'nski~\cite{barbour1989central}. For larger values of $t$, we split this difference using the triangle inequality and the problem is reduced into estimating $|\Ex{e^{itX^*}}|$ for $t\in [(p^2n)^{1/2},\pi\sigma]$, which is the core of this paper.	In this regime, we use a decoupling trick that allows us to get rid of some dependencies among triangles and  compare $|\Ex{e^{itX^*/\sigma}}|$ with the characteristic function of a sum of signed dilated binomials. 
	
	The rest of this paper is organised as follows.
	In Section~\ref{sec:proof-overview} we prove Theorems~\ref{thm:main} and~\ref{thm:main-2} from our main technical theorems, which analyse the integral of the characteristic function of $X$ across distinct regions of the real line, and we give an overview of their proof; Sections 3, 4 and 5 are dedicated to the proofs of these main technical theorems.

 }
\pagebreak
	\section{Proof of main results and overview}\label{sec:proof-overview}

	The first step in the proof of Theorem~\ref{thm:main} is the following theorem, which is a consequence of a result of Barbour, Karo\'nski and Ruci\'nski~\cite[Lemma 1]{barbour1989central} and the so-called Stein's method~\cite[Theorem 3.1]{pikachu}.
    See Section~\ref{sec:stein} for a proof.
    For convenience, we denote by $\N(x) = \frac{1}{\sqrt{2\pi}}e^{-x^2/2}$ the density of the standard normal random variable $N(0,1)$.

\begin{theorem}
	\label{theorem:first-regime-new}
		Let $X$ be the number of triangles in $G(n,p)$, with $\mu = \Ex{X}$ and $\sigma^2=\Var{X}$,
		and let $K=K(n)\ge 2$. 
		If  $p \in (n^{-1/2},1/2)$ we have
			  \[\sup_{x \in \LL }\big|\thinspace \N(x)-\sigma\cdot \Pr{X^{*} = x} \big| \leq 2\int\limits_{K}^{\pi \sigma}|\Ex{e^{itX^*}}|dt+2e^{-K}+ O\left(\frac{K}{np^{1/2}}\right),\]
			where $\mathcal{L}$ is the support of $X^{*}$.
	 \end{theorem}

	 In order to bound $\int\limits_{K}^{\pi \sigma}|\Ex{e^{itX^*}}|dt$, we split the integral into two parts and use our following two theorems, which are the main technical results of our paper.  
	 Theorem~\ref{thm:main-3} and Theorem~\ref{thm:main-4} shall be proven in Sections~\ref{sec:second-regime} and~\ref{section:end}, respectively.

	 \begin{theorem}\label{thm:main-3}
		Let $n \in \mathbb{N}$, $\gamma \in (0,1/8)$ and $p \in (4n^{-1/2}, 1/2)$. 
		For every $(2^{21}p^2n)^{1/2+\gamma}< t < \sigma/2^{10}$ we have
		\begin{align*}
			| \Ex{e^{itX/\sigma}}| \le \exp(-t^{2\gamma}).
		\end{align*}
	 \end{theorem}

        \begin{theorem}\label{thm:main-4}
		Let $n \in \mathbb{N}$ and $p\in (n^{-1/2},1/2)$.
		For every  $\sigma/2^{12} < t \le \pi \sigma$ we have
		\begin{align*}
			| \Ex{e^{itX/\sigma}}| \le \exp(-\Omega(\sqrt{n})).
		\end{align*}
        \end{theorem}

	Now we are ready to prove Theorem~\ref{thm:main}.

	\begin{proof}[Proof of Theorem~\ref{thm:main}]
    
	As mentioned before, a local central limit theorem for triangle counts was obtained by R\"ollin and Ross~\cite{rollin2015local} in the range where $n^{-1} \ll p = O(n^{-1/2})$ and extended by Gilmer and Kopparty~\cite{gilmer2016local} for constant $p \in (0,1)$.
    Thus, it suffices to show Theorem~\ref{thm:main} for $p \in (4n^{-1/2},1/2)$.

		Let $\gamma \in (0,1/8)$ and set $K := (\log n)^{\frac{8}{\gamma}} (p^2 n)^{\frac{1}{2}+\gamma }$.
		It follows from Theorem~\ref{thm:main-3}  and Theorem~\ref{thm:main-4} that 
		\begin{align}\label{eq:main-10}
			\int\limits_{K}^{\pi \sigma}|\Ex{e^{itX^*}}|dt & \le \int\limits_{K}^{ \sigma/2^{10}}|\Ex{e^{itX^*}}|dt + \sigma\exp(-\Omega(\sqrt{n})) \nonumber\\
			& \le \int\limits_{K}^{ \sigma/2^{10}} e^{-t^{2\gamma}} dt + \exp(-\Omega(\sqrt{n})) \nonumber \\
			& \le 
			\exp(-\Omega(\log n)^{8}).
		\end{align}
		As $e^{-K} \le \exp(-(\log n)^8)$ for all sufficiently large $n$, it follows from Theorem~\ref{theorem:first-regime-new} combined with~\eqref{eq:main-10}
		that
		\begin{align*}
			\sup_{x \in \LL }\big|\thinspace \N(x)-\sigma\cdot \Pr{X^{*} = x} \big|
			\le \exp(-\Omega(\log n)^{8}) + O \left(\dfrac{(\log n)^{\frac{8}{\gamma}} (p^2 n)^{\frac{1}{2}+\gamma }}{np^{\frac{1}{2}}} \right) \le n^{-\frac{1}{2}+2\gamma}p^{\frac{1}{2}+2\gamma}.
		\end{align*}
	\end{proof}

	Now we are ready to prove Theorem~\ref{thm:main-2}.

	\begin{proof}[Proof of Theorem~\ref{thm:main-2}]
		Let $\mathcal{L}$ be the support of $X^{*}=\frac{X-\mu}{\sigma}$, $\gamma \in (0,1)$ and $M = (\log n)^8$.
		We first change variables and split the sum into two parts:
		\begin{align}\label{eq:proof-main-result}
			&\sum \limits_{x\in \mathbb{N}}\left|\thinspace \frac{1}{\sqrt{2\pi \sigma^2}}e^{-\frac{(x-\mu)^2}{2\sigma^2}}- \Pr{X = x} \right| =\frac{1}{\sigma}
			\sum \limits_{\substack{x\in \LL }}\left|\thinspace \frac{1}{\sqrt{2\pi}}e^{-x^2/2} - \sigma\cdot\Pr{X^* = x} \right| \nonumber \\
			& \qquad \quad \le \frac{1}{\sigma}\sum \limits_{\substack{x\in \LL \\ |x| < M}}\left|\thinspace \N(x) - \sigma \cdot\Pr{X^* = x} \right| + 
			\sum \limits_{\substack{x\in \LL \\ |x| > M}} \dfrac{\N(x)}{\sigma} +\Pr{X^* \geq M}  \nonumber \\
			& \qquad \quad\le O(M  n^{-\frac{1}{2}+\gamma}p^{\frac{1}{2}+\gamma}) +  e^{-M} + \Pr{|X^*| \ge M}.
 		\end{align}
		In the last inequality, in the first term we used Theorem~\ref{thm:main} and that the sum has $O(\sigma M)$ terms.

		We now would like to apply Kim--Vu's inequality~\cite{kim2000concentration} to bound $\Pr{|X^*| \ge M}$ (see Theorem 3A in~\cite{janson2002infamous} or Theorem~\ref{thm:kimvu} below).
        Let $x_{e}$ denote the indicator function of the event that $e \in E(G(n,p))$.
        We can write $X$ as 
        \begin{align*}
            X = \sum \limits_{e,f,g} x_{e}x_{f}x_{g}
        \end{align*}
        where the sum is over all edges $e,f,g \in E(K_n)$ which form a triangle.
        For a set $A \se \{x_{e}: e \in E(K_n)\}$, note that the partial derivative $\partial_{A}(X)$ is a non-zero polynomial if and only if $A \se \{ x_{e},x_{f},x_{g}\}$, for some edges $e,f,g$ which form a triangle.
        By writing $e=uv$, $f=vw$ and $g=wu$, we obtain
        \[ \partial_{x_{e}}(X) = \sum_w x_{uw}x_{vw}, \quad \partial_{x_{e},x_{f}}(X) = x_{uw} \quad \text{and} \quad \partial_{x_{e},x_{f},x_{g}}(X) = 1.\]
        As $\partial_{\emptyset}(X)= X$, it follows from the definition of $\mathbb{E}_j$ in~\eqref{eq:partial-derivative} that 
        $\mathbb{E}_0(X) = O(n^3p^3)$ and $\mathbb{E}_1(X)=O(np^2)$.
        As the variance $\sigma^2$ of $X$ is of the same order as the expected number of $K_4^{-}$ when $p\ge n^{-1/2}$, we have $\sigma^2= \Theta(n^4p^5)$.
        This implies that $\mathbb{E}_0(X) \mathbb{E}_1(X) = O(\sigma^2)$, and hence it follows from Kim--Vu's inequality that 
			\begin{align}\label{eq:proof-main-result-2}
				\Pr{|X - \Ex{X}|\ge (\log n)^8 \sigma} \le
				e^{-\Omega((\log n)^2)}.
			\end{align}
			By combining \eqref{eq:proof-main-result} and \eqref{eq:proof-main-result-2}, we conclude that the left-hand side of \eqref{eq:proof-main-result} is $O(n^{-\frac{1}{2}+2\gamma}p^{\frac{1}{2}+\gamma})$.
	\end{proof}
\vspace*{-0.5cm}
 \subsection{Proof overview}
 Our strategy relies on relating the characteristic function of $X$ to that of a sum of independent, weighted Bernoulli random variables. Therefore, in order to reach higher values of $t$, we must relate $X$ to a sum of Bernoulli random variables with progressively smaller weights.
	
	This motivates the introduction of what we call the $\alpha$-function. Let $G^0$ and $G^1$ be two independent copies of $G(n,p)$. For and edge $e$ and $i\in \{0,1\}$ we write $x_e^i=\mathds{1}_{e\in G^i}$ and for a set of edges $A$, we write $x_A^i=(x_e^i)_{e\in A}$. 
    Let $m=m(n)$ be a positive integer and let
    $P_1 \cup P_2 \cup P_3$ be a partition of $[n]$ with $|P_1|= \lfloor (n-m)/2 \rfloor$, $|P_2|= \lceil (n-m)/2 \rceil$ and $|P_3|=m$.
    For simplicity, we set
	$A = \{uv: u \in P_1, v \in P_2\}$,
	$B_i = \{uv: u \in P_i, v \in P_3\}$ for $i \in \{1,2\}$ and $B=B_1\cup B_2$.
	We omit the dependencies of these sets on the partition, as they are clear from context.
	Finally, for $uv \in A$ define
	\begin{align*}
		\alpha_{uv} := \sum \limits_{w \in P_3}(x_{uw}^0-x_{uw}^1)(x_{vw}^0-x_{vw}^1).
	\end{align*}
	That is, $\alpha_f$ is a sum of $\pm 1$ terms over the common neighbours of $f$ in $G^0\Delta G^1$, the symmetric difference between $G^0$ and $G^1$.
	Note that the expression $x_{uw}^{0}-x_{uw}^{1}$ assigns value $+1$ to edges in $G^0\setminus G^1$ and $-1$ to edges in $G^1\setminus G^0$. Now, define
	\[\alpha = \sum \limits_{f\in A} \alpha_f x_f^0.\]
	In Lemma~\ref{lemma:CStrick-2} we use a \textit{decoupling trick} to show that
	\[|\Ex{e^{itX/\sigma }}|^4\le  \cE_{x_B^{0,1}} |\cE_{x_A^0} ( e^{it\alpha/\sigma } )|.\]
	
	Vaguely speaking, the $\pm 1$ signs from $\alpha$ come from replacing the square of the absolute value of a complex number by the product of itself with its conjugate.  We aim to show that a constant proportion of the edges $f\in A$ attain a typical value of $\alpha_f$. Even though $\alpha_f$ has expectation $0$, we show that with high probability 
	many $f\in A$ satisfy $|\alpha_f| = \Theta (pm^{1/2})$ as long as $pm^{1/2}$ is a large enough constant. 
	By conditioning on these events, we show inside the proof of Theorem \ref{thm:main-3} that
	\[|\cE_{x_A^0}(e^{it\alpha/\sigma})|\leq \exp\left(-\Omega(t^2m/p^2n^2)\right).\]
	For this to be integrable with respect to $t$, $m$ has to be slightly larger than $p^2n^2/t^2$. 
     We emphasize that $m$, which is the size of the part $P_3$, can be any parameter between $1$ and $n$ and is allowed to be a function of $t$.
     On the other hand, we also ask that $t\alpha_f/2\pi \sigma \leq 1/2$, as argued before, or equivalently that $t= O(\sigma/m^{1/2})$. By decreasing the order of magnitude of $m$ from $n$ to $1/p^2$ we are then able to show in Theorem \ref{thm:main-3} that
	\[|\Ex{e^{itX/\sigma}}|\leq \exp(-t^{\Theta(1)})\]
	\noindent for all $(p^2n)^{1/2+\gamma}<t< \sigma/2^{10}$, for any $\gamma>0$.
	
	We are only left with the case $t=\Theta(\sigma)$. To be able to reach the value of $t$ up to $\pi \sigma$ we still use the $\alpha$ function, but we must create a situation where the typical value of $|\alpha_f|$ is $1$. We then explore the \textit{role of a single vertex}, i.e. the case where $P_3=\{v\}$. In this case, $\alpha$ is just a signed sum over the edges in $N(v,G^0)\Delta N(v,G^1)$. Since this set has roughly $pn$ vertices, each of the $\Theta((pn)^2)$ pairs inside of it are candidates for having its $|\alpha_f|=1$. Therefore, $\alpha$ becomes a sum of roughly $(pn)^2$ Bernoulli random variables with $\pm 1$ weights. We use that in the proof of Theorem \ref{thm:main-4} to show that
	\[|\Ex{e^{itX\sigma}}|\leq \exp(-\Omega(\sqrt{n})),\]
	\noindent for all $\sigma/2^{12} \leq t \leq \pi \sigma$, which finishes the sketch of proof of Theorem~\ref{thm:main-4}.
	
	\section{Toolbox}\label{sec:tool-box}
	
	In this section, we state the probabilistic tools used in the proof of our main theorem. The first one gives us a bound on the characteristic function of Binomial random variables. 
	For a proof, see Lemma 1 in~\cite{gilmer2016local}.
	Below, $\|x\|$ denotes the distance from $x$ to the nearest integer.

	\begin{lemma}\label{lemma:charfuncbinomial}
		Let $p \in [0,1]$, $n \in \mathbb{N}$ and $t \in \mathbb{R}$. If $B$ is a $p$-Bernoulli random variable, then
		\[|\Ex{e^{itB}}| \le 1-8p(1-p)\cdot \left \| \dfrac{t}{2\pi} \right \|^2 .\]
	\end{lemma}
	
	The next lemma states the well-known Chernoff's inequality. For a proof, see Theorem 23.6 in~\cite{frieze2016introduction}.
	
	\begin{lemma}[Chernoff's inequality]\label{lemma:chernoff}
		Let $Y$ be a binomial random variable. For every $t \ge 0$, we have  
		\[\Pr{|Y-\Ex{Y}|\ge t}\le 2e^{-t^2/(2\Ex{Y}+t)}.\]
	\end{lemma}
	

%
	
	Let $P$ be a random polynomial in terms of the indicators $\{\mathds{1}_{e \in G}: e \in K_n\}$.
	For a set $A \se \binom{[n]}{2}$, let $\partial_A P $ be the partial derivative of $P$ with respect to the variables in $A$. When $A = \emptyset$, we simply set $\partial_A P = P$.
	For $j \in \{0,1,\ldots,\text{deg}(P)\}$, define 
	\begin{align}\label{eq:partial-derivative}
	    \cE_j(P):= \max_{|A|\ge j} \Ex{\partial_A P}.
	\end{align}
	The theorem of Kim and Vu~\cite{kim2000concentration} for random polynomials of degree three can be stated as follows (see Theorem 3A in~\cite{janson2002infamous}).
	\begin{theorem}[Kim--Vu's inequality]\label{thm:kimvu}
		Let $P$ be a random polynomial of degree $d$ whose variables belong to the set $\{\mathds{1}_{e \in G}: e \in K_n\}$.
		For every $r > 1$, we have
		\begin{align*}
			\Pr{|P - \Ex{P}|\ge c_d r^d\big (\cE_0(P)\cE_1(P)\big )^{1/2}} \le
			e^{-r+(d-1)\log \binom{n}{2}}.
		\end{align*}
	\end{theorem}
	
	Our last ingredient is the Paley--Zygmund inequality~\cite{paley1932note}, which gives an lower bound for the probability that a random variable is greater than a constant proportion of its expectation.
	
	\begin{theorem}[Paley--Zygmund inequality]\label{thm:paley}
		Let $\theta \in (0,1)$ and $S \ge 0$ be a random variable with finite variance. Then,
		\begin{align*}
			\Pr{S > \theta \Ex{S}}\ge (1-\theta)^2 \dfrac{(\Ex{S})^2}{\Ex{S^2}}.
		\end{align*} 
	\end{theorem}

 \section{Stein's Method}\label{sec:stein}


 
     In this section we combine some results surrounding the so-called Stein's Method~\cite{pikachu} in the particular case of triangle counts in $G(n,p)$ and use it to prove Theorem \ref{theorem:first-regime-new}. We define $\mathcal{W} := \{ h:\mathbb{R}\rightarrow \mathbb{R} : |h(x)-h(y)|\leq |x-y| \text{ }\forall x,y\in \mathbb{R}\}$ to be the set of 1-Lipschitz functions. For random variables $X,Y$ we define 
     \[d_\mathcal{W}(X,Y):= \sup_{h\in \mathcal W} \left|\mathbb{E}\big(h(X)-h(Y)\big)\right|\]
     \noindent and call it the \emph{Wasserstein metric}.

 \begin{lemma}{\cite[Theorem 3.1]{pikachu}}
 \label{lemma:Stein-lemma}
    Let $\mathcal{F}=\{f:\mathbb{R}\rightarrow \mathbb{R} : |f|,|f''|\leq 2, |f'|\leq \sqrt{2/\pi} \}$. If $X$ is a random variable and $N$ is the standard normal distribution, then
        \[d_\mathcal{W}(X,N)\leq \sup_{f\in \mathcal{F}}\left|\mathbb{E}\big(Xf(X) -f'(X)\big)  \right|.\]
 \end{lemma}

  Barbour, Karo\'nski and Ruci\'nski~\cite{barbour1989central} proved a quantitative central limit theorem for subgraph counts in $G(n,p)$ using the Stein's Method. In one of their intermediate results, they give an upper bound on the right-hand side of the inequality given by Lemma~\ref{lemma:Stein-lemma}, which we state below.

 \begin{lemma}{\cite[Lemma 1]{barbour1989central}}
\label{lemma:new-rucinski}
     Let $X$ be the number of triangles in $G(n,p)$. If $n^{-\frac{1}{2}} \le p \le 1/2$ then for every bounded function $f:\mathbb{R}\rightarrow \mathbb{R}$ with bounded first and second derivatives we have
     \[\left|\mathbb{E}\big(X^*f(X^*) -f'(X^*)\big)  \right| = O\left(\frac{B}{np^{1/2}}\right) ,\]
     \noindent where $B=\sup_x|f''(x)|$.
 \end{lemma}

Let us just briefly explain how to derive our Lemma~\ref{lemma:new-rucinski} from the results in~\cite{barbour1989central}.
By combining Lemma 1 in~\cite{barbour1989central} with the upper bound on $\varepsilon$ for triangle counts given by their equation (3.10), we obtain the following result. For every $n^{-1/2} \le p \le 1/2$ and every bounded function $f: \mathbb{R} \to \mathbb{R}$ with bounded first and second derivatives, we have
\begin{align}\label{eq:stein-barbour}
    \left|\mathbb{E}\big(X^*f(X^*) -f'(X^*)\big)  \right| = O\left(\frac{Bn^3p^3}{\sigma \psi}\right) ,
\end{align}
where $B=\sup_x|f''(x)|$, $\sigma^2 = \Var{X}$ and $\psi = \min \{n^{v(H)}p^{e(H)}: H \se K_3, e(H) \ge 1\}$.
Recall that $X^{*}$ denotes the random variable $(X-\Ex{X})/\sigma$.
It is not hard to check that if $n^{-\frac{1}{2}} \le p \le 1/2$, then $\psi = \Theta(n^2p)$ and $\sigma = \Theta(n^2p^{\frac{5}{2}})$, and hence the right-hand size of~\eqref{eq:stein-barbour} is $O(B/(np^{1/2}))$.


We are now ready to prove Theorem \ref{theorem:first-regime-new}.
 \begin{proof}[Proof of Theorem \ref{theorem:first-regime-new}]
We start with the well-known Fourier inversion formula for $X^*$ (for a proof, see \cite[Page 511]{feller}), which is
		\[\Pr{X^{*}=y} = \dfrac{1}{2\pi \sigma}  \int \limits_{-\pi \sigma}^{\pi \sigma} e^{-ity}\Ex{e^{itX^*}}dt\]
		
		\noindent for every $y$ in the support of $X^{*}$.
		It is well-known that
		\[\N(x)=\frac{1}{2\pi} \int\limits_{-\infty}^{\infty}e^{-itx}e^{-t^2/2}dt.\]
	This can be seen by using that the characteristic function of the standard normal random variable $N(0,1)$ is $e^{-t^2/2}$ (\cite{durrett}, example 3.3.5) together with the Fourier inversion formula for variables with integrable characteristic function over $\mathbb{R}$ (\cite{durrett}, Theorem 3.3.14).
	
	Now we start the analysis on the main quantity of interest:
	\begin{align}\label{eq:difference}
	    \big|\thinspace \N(x)-\sigma\cdot \Pr{X^{*} = x} \big| \leq \int\limits_{-\pi \sigma}^{\pi \sigma} |\Ex{e^{itX^*}}-e^{-t^2/2}|dt  + 2\int\limits_{\pi \sigma}^{\infty} e^{-t^2/2}dt.
	\end{align}
We now split the first term in right-hand side of~\eqref{eq:difference} according to the values of $K$. For $|t|> K$ we use the triangle inequality and the fact that $|\Ex{e^{itX^*}}|$ and $e^{-t^2/2}$ are even functions to obtain
\begin{equation} 
\label{eq:1}
\big|\thinspace \N(x)-\sigma\cdot \Pr{X^{*} = x} \big| \leq 
2\int\limits_{K}^{\pi \sigma}|\Ex{e^{itX^*}}|dt + \int\limits_{-K}^{K}|\Ex{e^{itX^*}}-e^{-t^2/2}|dt + 2\int\limits_{K}^{\infty}e^{-t^2/2}dt.
\end{equation}
For the third term in the right-hand side of the equation above we have
\begin{equation}
    \label{eq:2}
	\int\limits_{K}^{\infty} e^{-t^2/2}dt\leq \int\limits_{K}^{\infty} \frac{t}{K} e^{-t^2/2}dt = \dfrac{e^{-K^2/2}}{K} \le e^{-K}
\end{equation}
for all $K\ge 2$.

The characteristic function of the normal distribution is given by $\Ex{e^{itN}}=e^{-t^2/2}$. 
By using this and splitting $\Ex{e^{itX^*}}$ and $\Ex{e^{itN}}$ into its real and imaginary parts, we can bound the second term in the right-hand side of~\eqref{eq:1} by
\begin{align*}
    \int\limits_{-K}^{K} |\Ex{e^{itX^*}}-e^{-t^2/2}|dt \leq \int\limits_{-K}^{K} \big|\mathbb{E}\big(\text{cos}(tX^*)-\text{cos}(tN)\big)\big|dt + \int\limits_{-K}^{K} \big|\mathbb{E}\big(\text{sin}(tX^*)-\text{sin}(tN)\big)\big|dt.
\end{align*}

As the cosine and sine functions are 1-Lipschitz, for every $t \in \mathbb{R}$
we have 
\[ |\Ex{h(tX^{\ast})-h(tN)}| \le d_\mathcal{W}(tX^{\ast},tN)\]
for both $h = \cos$ and $h = \sin$. 
Moreover, as cosine and sine have second derivatives bounded by $1$, we can combine Lemmas \ref{lemma:Stein-lemma} and \ref{lemma:new-rucinski} to obtain
\begin{equation}
    \label{eq:3}
    \int\limits_{-K}^{K} |\Ex{e^{itX^*}}-e^{-t^2/2}|dt = O\left(\frac{K}{np^{1/2}}\right).
\end{equation}
Combining the bounds obtained in \eqref{eq:2} and \eqref{eq:3} into \eqref{eq:1} we conclude our proof.
 \end{proof}

	\section{The decoupling trick in the regime $ (p^2n)^{1/2+\gamma}\leq t \leq \sigma/2^{10}$} \label{sec:second-regime}

    In this section, we prove Theorem~\ref{thm:main-3}. 
    Our approach is to first bound $\left|\mathbb{E}\left[e^{itX/\sigma}\right]\right|$ by relating it to the characteristic function of certain weighted Bernoulli random variables, and then applying Lemma~\ref{lemma:charfuncbinomial}. 
    Since we expect a constant proportion of the edges in $G(n,p)$ to participate in $\Theta(np^2)$ triangles, one might initially hope that by revealing some edges of $G(n,p)$, we could bound $\left|\mathbb{E}\left[e^{itX/\sigma}\right]\right|$ via the characteristic function of a random variable of the form $\sum_{j=1}^{cn^2} C_j \mathds{1}_{\{e_j \in G(n,p)\}},$ where the coefficients $C_j$ are of order $\Theta\left(\frac{np^2}{\sigma}\right)$. However, this method yields a satisfactory bound only for $|t| < \sigma/(np^2)$. To address this limitation, we introduce a decoupling trick that connects the triangle count to a finite sum of ``signed triangles''. After revealing some of the variables, we are left with a sum of weighted Bernoulli random variables with small coefficients.

	\subsection{The $\alpha$ function} 
	
	We start with some notation. Let $G^0$ and $G^1$ be two independent copies of $G(n,p)$ with vertex set $[n]$.
	For $i \in \{0,1\}$ and a set $A\se \binom{[n]}{2}$, denote by $x^i_A$ the vector $(\mathds{1}_{e \in G^i}: e \in A)$ and by $x^{0,1}_A$ the concatenation of the vectors $x^0_A$ and $x^1_A$. 
	For a 0-1 random vector $y$,
	we write $\cE_{y}$ to denote the expectation operator with respect to $y$.
	If $y$ equals $x^0_{\binom{[n]}{2}}$ or $x^1_{\binom{[n]}{2}}$, we omit $y$ from the expectation operator.
	Finally, for a 0-1 vector $x$ indexed by $\binom{[n]}{2}$,
    we denote
    \begin{align*}
        X(x) = \sum \limits_{e,f,g} x_e x_f x_g,
    \end{align*}
	where the summation runs over all triples $e, f, g \in \binom{[n]}{2}$ that form a triangle.
	
	Let $m=m(n)$ be a positive integer.
	We say that a partition $[n]= P_1 \cup P_2 \cup P_3$ with $|P_1|= \lfloor (n-m)/2 \rfloor$, $|P_2|= \lceil (n-m)/2 \rceil$ and $|P_3|=m$ is an $m$-\textit{endowed partition} of $[n]$.
	For simplicity, we set
	$A = \{uv: u \in P_1, v \in P_2\}$,
	$B_i = \{uv: u \in P_i, v \in P_3\}$ for $i \in \{1,2\}$ and $B=B_1\cup B_2$.
	We omit the dependencies of these sets on the partition, as they are clear from context.
	Finally, for $uv \in A$ define
	\begin{align}\label{eq:defalphaf}
		\alpha_{uv} := \sum \limits_{w \in P_3}(x_{uw}^0-x_{uw}^1)(x_{vw}^0-x_{vw}^1)
	\end{align}
and
	\begin{align}\label{eq:defalpha}
		\alpha = \sum \limits_{f\in A} \alpha_f x_f^0.
	\end{align}
    We emphasize that $\alpha$ does not depend on $x^1_A$.
	
	The next lemma uses $\alpha$ to bound $|\Ex{e^{itX/\sigma }}|$ and is essentially a corollary of Lemma 6 in~\cite{berkowitz2018local}.
	We include the proof here for completeness.
	The idea of using two independent copies of $G(n,p)$ to bound $|\Ex{e^{itX/\sigma }}|$ was introduced by Gilmer and Kopparty~\cite{gilmer2016local} and it has been used since then in local central limit theorems for subgraph counts (see~\cite{berkowitz2018local} and~\cite{sah2022local}).
	This is what we refer to as a \emph{decoupling trick}.
	Our innovation is to apply this technique by bounding $|\Ex{e^{itX/\sigma }}|$ in terms of $\alpha$ and by varying $m$ from $1/p^2$ to $n/2$,
    thereby covering a wider range and dealing with the technical aspects of the sparse regime of $G(n,p)$.

	\begin{lemma}\label{lemma:CStrick-2}
		Let $m < n$ be two positive integers and $P_1\cup P_2 \cup P_3$ be an $m$-endowed partition of $[n]$.
		For every $t \in \mathbb{R}$ and $p=p(n)\in (0,1)$, we have
		\[\left|\Ex{e^{itX/\sigma }}\right|^4\le  \cE_{x_B^{0,1}} \left|\cE_{x_{A}^0} ( e^{it\alpha/\sigma } )\right|.\]
	\end{lemma}

      As $\alpha$ does not depend on $x^1_A$, both sides of the inequality in the lemma above are fixed quantities and not functions of any random variables.
	
	\begin{proof}[Proof of Lemma~\ref{lemma:CStrick-2}]
		By the Cauchy--Schwartz inequality, we have
		\begin{align}\label{eq:sc1}
			\left|\Ex{e^{itX(x^0)/\sigma }}\right|^2 & \le \cE_{x_{B_1^c}^0} |\cE_{x_{B_1}^0} ( e^{it X(x^0)/\sigma })|^2.
		\end{align}
		Now, observe that we can write
		\begin{equation}
			\label{eq:sc2}
			\begin{split}
				\left|\cE_{x_{B_1}^0} ( e^{it X(x^0) /\sigma})\right|^2 & = 
				\cE_{x_{B_1}^0} \left( e^{it X(x^0) /\sigma}\right) \cdot \cE_{x_{B_1}^0} \left( e^{-it X(x^0) /\sigma}\right)\\
				&= \cE_{x_{B_1}^0} \left( e^{it X(x^{0}_{B_1^c},x_{B_1}^{0})/\sigma}\right) \cdot \cE_{x_{B_1}^1} \left( e^{-it X(x_{B_1^c}^{0},x_{B_1}^{1})/\sigma}\right) \\
				& = \cE_{x_{B_1}^{0,1}} \left( e^{it( X(x_{B_1^{c}}^{0},x_{B_1}^{0}) - X(x_{B_1^{c}}^{0},x_{B_1}^{1}))/\sigma}\right),
			\end{split}
		\end{equation}
        where we denote $B_1^c \coloneqq \binom{[n]}{2} \setminus B_1$.
		In the second equality above, we used that $x^0_{B_1}$ and $x^1_{B_1}$ independent and identically distributed (i.i.d).
		In the third equality, we used that $x^0_{B_1}$ and $x^1_{B_1}$ are i.i.d. combined with the fact that $x_{B_1^{c}}^{0}$ does not depend on neither $x^0_{B_1}$ nor $x^1_{B_1}$.
		To simplify notation, let $\beta \coloneqq X(x_{B_1^{c}}^{0},x_{B_1}^{0}) - X(x_{B_1^{c}}^{0},x_{B_1}^{1})$. 
		By combining~\eqref{eq:sc1} and~\eqref{eq:sc2}, we obtain
		\begin{equation}\label{eq:firstsquare}
			\begin{split}
				\left|\Ex{e^{itX(x^0)/\sigma }}\right|^2 & \le \cE_{x^0_{B_1^c}} \cE_{x_{B_1}^{0,1}} ( e^{it \beta/\sigma })\\
				& = \cE_{x_{B_1}^{0,1}} \cE_{x^0_{B_1^c}} ( e^{it \beta/\sigma }).
			\end{split}
		\end{equation}
		
		By representing $\beta$ as a sum of monomials, one can see that the monomials corresponding to triangles without edges in $B_1$ are eliminated.
		Moreover, each triangle of the form $uvw$ with $v \in P_1$, $u \in P_2$ and $w \in P_3$ gives a contribution of $x_{uv}^{0}x_{vw}^{0}(x^0_{vw}-x^1_{vw})$ to $\beta$.
		By simplicity, let $C= \binom{P_1}{2} \cup B_1 \cup \binom{P_3}{2}$.
		Thus, we can write
		\begin{align}\label{eq:gamma}
			\beta = P(x_{C}^{0,1})+
			\sum \limits_{\substack{v \in P_1,\\ u \in P_2,\\ w \in P_3}}x_{uv}^{0}x_{uw}^{0}(x^0_{vw}-x^1_{vw}),
		\end{align}
		where $P$ is some polynomial.
		
		By applying the Cauchy--Schwartz inequality in~\eqref{eq:firstsquare}, we obtain
		\begin{align}\label{eq:almostthere}
			\left|\Ex{e^{itX(x^0)/\sigma }}\right|^4 & \le
			\cE_{x_{B_1}^{0,1}} \left|\cE_{x^0_{B_1^c}} ( e^{it \beta/\sigma })\right|^2\nonumber\\
			& \le
			\cE_{x_{B_1}^{0,1}} \cE_{x^0_{B_1^c \cap B_2^c}} \left|\cE_{x^0_{B_2}} ( e^{it \beta/\sigma })\right|^2.
		\end{align}
        Define $\Tilde{\beta}$ as the polynomial $\beta$ with the vector $x_{B_2}^0$ replaced by $x_{B_2}^1$.
        Thus, we have
        \begin{align}\label{eq:squaregamma}
			\left|\cE_{x^0_{B_2}} ( e^{it \beta/\sigma })\right|^2 =  \cE_{x^0_{B_2}} ( e^{it \beta/\sigma }) \cdot  \cE_{x^0_{B_2}} ( e^{-it \beta/\sigma }) = \cE_{x^{0,1}_{B_2}} ( e^{it (\beta-\Tilde{\beta})/\sigma }).
		\end{align}
        Similarly as before, the difference $\beta-\Tilde{\beta}$
		eliminates all monomials which do not depend on $B_2$.
		In particular, $P(x_C^{0,1})$ completely disappears.
		From~\eqref{eq:gamma} we obtain that
		\begin{align}\label{eq:gammadiff}
			\beta-\Tilde{\beta}&=
			\sum \limits_{\substack{v \in P_1,\\ u \in P_2,\\ w \in P_3}}x_{uv}^{0}(x_{uw}^{0}-x_{uw}^{1})(x^0_{vw}-x^1_{vw})= \alpha.
		\end{align}
		Note that $\alpha$ only uses randomness from $x_{B_1}^{0,1}$, $x_{B_2}^{0,1}$ and $x_A^0$.
        As $B = B_1 \cup B_2$ and $A \se B_1^c \cap B_2^c$,
		by combining~\eqref{eq:almostthere},~\eqref{eq:squaregamma} and~\eqref{eq:gammadiff} we obtain
		\begin{equation*}
			\begin{split}
				\left|\Ex{e^{itX(x^0)/\sigma }}\right|^4 \le 
				\cE_{x_{B_1}^{0,1}} \cE_{x^0_{B_1^c \cap B_2^c}} \cE_{x_{B_2}^{0,1}} (e^{it\alpha/\sigma})
				\le \cE_{x_{B}^{0,1}}|\cE_{x_A^0}(e^{it\alpha/\sigma})|. \nonumber
			\end{split}
		\end{equation*}
        This finishes our proof.
	\end{proof}
	\vspace*{-1cm}
 \setstretch{1.05}{
	\subsection{On the typical values of $\alpha$} \label{sec:alpha}

	
	Throughout this subsection, all probabilities will be with respect to $x_B^0$ and $x_B^1$, where we recall that $B = \{vw: v \in P_1 \cup P_2 \text{ and } w \in P_3\}$.
	Let us start by calculating $\Ex{\alpha_{uv}^2}$ for an edge $uv \in A$.
	To simplify notation, we denote $f=uv$ and for $j \in P_3$ we write
	\[X_j:= (x_{uj}^0-x_{uj}^1)(x_{vj}^0-x_{vj}^1) \qquad \text{and} \qquad \alpha_f = \sum \limits_{j\in P_3} X_j.\]	
	Note that $\Ex{X_j} = 0 $ and $X_j,X_{j'}$ are independent whenever $j\neq j'$. Moreover, $X_j^2 = |X_j|$, as $X_j$ can only take values in $\{-1,0,1\}$. Then, we have
	\[\Ex{\alpha_f^2} = \sum_{j\in P_3} \Ex{X_j^2}= \Pr{X_j \in \{1,-1\}}\cdot m= 4p^2(1-p)^2m.\]
	
	The value of $\alpha_f^2$ is closely related to the number of common neighbours of a pair of vertices in $G(n,p)$. 
	As we allow $p$ to be of order $n^{-1/2}$,
	it is not true that with high probability we have $\alpha_f^2 = \Theta(p^2m)$ for every $f \in A$.
	However we are able to show that as long as $m = \Omega(p^{-2})$ with very high probability we have $\alpha_f^2 = \Theta(p^2m)$ for a constant proportion of edges in $A$. 
	}
	\begin{lemma}
		\label{lemma:typical_alpha}
		Let $n, m \in \mathbb{N}$ and $p = p(n)\in (0,1)$ be such that
		$p^{-2}(1-p)^{-2}\le m \le n/2$.
		Let $P_1 \cup P_2 \cup P_3$ be an $m$-endowed partition of $[n]$ and $A'\subset A$ be the set of pairs $f$ such that \[\alpha_f \in \Big(2^{-1}\sqrt{\Ex{\alpha_f^2}}, 2^{3}\sqrt{\Ex{\alpha_f^2}}\Big).\] With probability at least $1-\exp (-\Omega(n))$ we have $|A'| \geq |A|/2^{7}$.
	\end{lemma}
	\vspace*{-0.3cm}
	\begin{proof}
		We start by providing a lower bound for the probability that $\alpha_f^2 \ge \Ex{\alpha_f^2}/2$.
		We have
		\[\Ex{\alpha_f^2} = \sum_{j\in P_3} \Ex{X_j^2}= \Pr{X_j \in \{1,-1\}}\cdot m= 4p^2(1-p)^2m.\]
		By the Payley--Zigmund inequality (c.f. Theorem~\ref{thm:paley}), we have
		\begin{equation}
			\label{eq:PZ-alpha}
			 \Pr{\alpha_f^2 \ge \dfrac{\Ex{\alpha_f^2}}{2}}
			\ge \dfrac{(\Ex{\alpha_f^2})^2}{4\Ex{\alpha_f^4}}= \frac{4p^4(1-p)^4m^2}{\Ex{\alpha_f^4}}.
		\end{equation}
		
		As the random variables $(X_j)_{j \in P_3}$ are independent and 
		$\Ex{X_j^2}=\Ex{X_j^4}=4p^2(1-p)^2$, the fourth moment of $\alpha_f$ is given by	
		\begin{align*}
			\Ex{\alpha_f^4} &= 6\sum \limits_{i,j \in \binom{P_3}{2}}\Ex{X_i^2 X_j^2}+ \sum \limits_{j\in P_3} \Ex{X_j^4}\\
			& = 6\binom{m}{2}(\Ex{X_1^2})^2+m\Ex{X_1^4}\\
			& \le 48 p^4(1-p)^4 m^2 + 4p^2(1-p)^2m \le 2^6 p^4(1-p)^4 m^2.
		\end{align*}
		In the last inequality, we used that $p^2(1-p)^2m\ge 1$.
		By plugging this into \eqref{eq:PZ-alpha}, we obtain
		\begin{align}\label{eq:payleyforalpha}
			\Pr{\alpha_f^2 \ge \dfrac{\Ex{\alpha_f^2}}{2}} \ge 2^{-4}.
		\end{align}
        Recall the definition of $A'$ provided in the statement.
		By Markov's inequality we have $\Pr{\alpha_f^2\ge 2^{5} \Ex{\alpha_f^2}} \le 2^{-5}$.
		We combine it with~\eqref{eq:payleyforalpha} to get that
		\begin{align}\label{eq:probfinA}
			\Pr{f\in A'}\ge \Pr{\dfrac{\Ex{\alpha_f^2}}{2} \le \alpha_f^2 \le 2^5\Ex{\alpha_f^2} } \ge 2^{-4}-2^{-5} = 2^{-5}.
		\end{align}
		
		Recall that $A = \{uv: u \in P_1, v \in P_2\}$, where $|P_1| = \lfloor (n-m)/2 \rfloor$ and $|P_2| = \lceil (n-m)/2 \rceil$.
		Let $M_1, M_2, \ldots, M_{|P_1|}$ be $|P_1|$ disjoint matchings of size $|P_1|$ in $A$.
		For each $i \in [|P_1|]$, define 
		$Z_i = \sum_{f \in M_i} \mathds{1}_{\{f\in A'\}}$.
		Observe that for each $i \in [|P_1|]$ the collection $(\mathds{1}_{\{f\in A'\}})_{f \in M_i}$ is composed by independent random variables.
		Then, by Chernoff's inequality (c.f. Lemma~\ref{lemma:chernoff}), we have
		\[\Pr{Z_i \le \Ex{Z_i}/2}\le \exp(-2^{-5}\Ex{Z_i})\le \exp(-2^{-13}n)\]
		for all $i \in [|P_1|]$ and $n$ sufficiently large.
		The last inequality follows from~\eqref{eq:probfinA} combined with $|P_1|\ge (n-m)/4\ge n/8$.
		As $\Ex{Z_i}\ge 2^{-5}|P_1|$, it follows that
		\begin{align*}
			|A'| \ge Z_1+\cdots+Z_{|P_1|}\ge 2^{-5}|P_1|^2 
		\end{align*}
		with probability at least $1-n\exp(-2^{-13}n)$.
		This proves our lemma.
	\end{proof}

	\subsection{Putting pieces together}
	
	We now finalise the argument by seeing $\alpha$ as a sum of independent Bernoullis with weight $\Theta(\sqrt{p^2m})$, after revealing the edges in $B$.
	\begin{proof}[Proof of Theorem~\ref{thm:main-3}]
		Let $p\in(4n^{-1/2},1/2)$ and $m$ be any integer such that $4/p^2 \le m \le n/2$.
		By Lemma~\ref{lemma:CStrick-2}, we have
		\begin{align}\label{eq:boundingcharmiddle}
			|\Ex{e^{itX/\sigma }}|^4\le  \cE_{x_B^{0,1}} |\cE_{x_{A}^0} ( e^{it\alpha/\sigma } )|.
		\end{align}
		Define
		\[A'=A'(x_B^{0,1}) := \left\{f\in A : \alpha_f \in \left(\frac{\sqrt{p^2m}}{2}, 2^{4}\sqrt{p^2m}\right)\right\}.\]
		As the set $A'$ only depends on the outcome $x_{B}^{0,1}$, we have
		\begin{align}\label{eq:expAprime1}
			|\cE_{x_A^0} ( e^{it\alpha/\sigma } )|  
            = |\cE_{x_{A\setminus A'}^{0}} \cE_{x_{A'}^0} ( e^{it\alpha/\sigma } )|
            \le
            \cE_{x_{A\setminus A'}^{0}}|\cE_{x_{A'}^0} ( e^{it\alpha/\sigma } )|
			 = \left |\cE_{x_{A'}^0} \Big ( e^{it/\sigma \cdot \sum \limits_{f \in A'} \alpha_f x_f^{0} } \Big ) \right |.
		\end{align}
		In the last equality, we used that $\alpha = \sum_{f \in A}\alpha_f x_f^0$ and that the random variables $(x_f^0)_{f \in A}$ are independent, with coefficients $(\alpha_f)_{f \in A}$ not depending on $x_A$.
		By Lemma~\ref{lemma:charfuncbinomial}, it follows that
		\begin{align*}
			\left |\cE_{x_{A'}^0} \Big ( e^{it/\sigma \cdot \sum \limits_{f \in A'} \alpha_f x_f^{0} } \Big ) \right | \le 
			\prod \limits_{f \in A'} \left(1-8p(1-p)\left \| \dfrac{t\alpha_f}{2\pi \sigma} \right \|^2 \right).
		\end{align*}
		
		As $\alpha_f \le 2^4\sqrt{p^2m}$ for every $f \in A'$, we have $\|t\alpha_f/(2\pi \sigma)\|=t\alpha_f/(2\pi \sigma)$ whenever $0 \le t \le \pi \sigma/(2^4\sqrt{p^2m})$.
		As $p \le 1/2$, it follows that
		\begin{align}\label{eq:expAprime2}
			\left |\cE_{x_{A'}^0} \Big ( e^{it/\sigma \cdot \sum \limits_{f \in A'} \alpha_f x_f^{0} } \Big ) \right | 
			&\le 
			\exp  \left( -\sum \limits_{f \in A'}  \dfrac{pt^2\alpha_f^2}{\pi^2 \sigma^2} \right) \nonumber\\
			& \le \exp  \left( - \dfrac{t^2p^3m|A'|}{4\pi^2 \sigma^2} \right)
		\end{align}
		for every $0 \le t \le \pi \sigma/(2^4\sqrt{p^2m})$.
		In the last inequality, we used that $\alpha_f^2 \ge p^2m/4$ for every $f \in A'$.
		By combining~\eqref{eq:expAprime1} and~\eqref{eq:expAprime2}, we obtain
		\begin{align}\label{eq:xazero}
			|\cE_{x_A^0} ( e^{it\alpha/\sigma } )| \le  \exp  \left( - \dfrac{t^2p^3m|A'|}{4\pi^2 \sigma^2} \right) 
			\le \exp  \left( - \dfrac{t^2m}{2^{19} p^2n^2} \right)+\mathds{1}_{\{|A'|<2^{-13}n^2\}}
		\end{align}
		for every $0 \le t \le \pi \sigma/(2^4\sqrt{p^2m})$. In the last inequality, we used that $\sigma^2 \le n^4p^5$ for $p \ge 4n^{-1/2}$.
		
		As $4n^{-1/2}\le p \le 1/2$, we have $\Ex{\alpha_f^2} \in (p^2m,4p^2m)$.
		Moreover, as $4/p^2 \le m \le n/2$, Lemma~\ref{lemma:typical_alpha} implies that the set $A'$ has size at least $|A'|\ge 2^{-7}|A|\ge 2^{-13}n^2$ with probability at least $1-\exp(-\Omega(n))$.
		By combining with~\eqref{eq:boundingcharmiddle} with~\eqref{eq:xazero}, it follows that
		\begin{align}\label{eq:4thpowercharfunc}
			|\Ex{e^{itX/\sigma }}|^4\le \exp  \left( - \dfrac{t^2m}{2^{19} p^2n^2} \right) + \exp(-\Omega(n))
		\end{align}
		for every $0 \le t \le \pi \sigma/(2^4\sqrt{p^2m})$.
		
		Let $\delta\in (0,2/5)$ be a constant. The next step is to ensure that the right-hand side of~\eqref{eq:4thpowercharfunc} decays `fast' if $ (p^2n)^{\frac{1}{2-\delta}}<t<\sigma/2^8$.
		With this in mind, for each integer $m \in (4/p^2,n/2)$ we define the interval
		\[I_m := \left(\left(\dfrac{2^{19}p^2n^2}{ m}\right)^{\frac{1}{2-\delta}}, \frac{\sigma}{2^8pm^{1/2}} \right)\]
		and set $I = \cup_{m} I_m$, where the union is over $m \in (4/p^2,n/2)\cap \mathbb{N}$.
		As $(\alpha +\beta)^{1/4}\le \alpha^{1/4}+\beta^{1/4}$ for every $\alpha, \beta \ge 0$, 
		it follows from~\eqref{eq:4thpowercharfunc} that for every $t \in I $ we have 
		\begin{align*}
			|\Ex{e^{itX/\sigma }}|\le \exp  \left( -t^{\delta} \right) + \exp(-\Omega(n)).
		\end{align*}
		
		Now, we claim that $\big((2^{21}p^2 n)^{\frac{1}{2-\delta}},\sigma/2^{10}\big) \se I$.
		First, note that the right-hand endpoint of $I_{\lceil 4/p^2 \rceil}$ is greater than $\sigma/2^{10}$ and the left-hand endpoint of $I_{\lfloor n/2 \rfloor}$ is smaller than $(2^{21}p^2 n)^{\frac{1}{2-\delta}}$.
		As the extreme values of $I_m$ are decreasing in $m$, it suffices to show that 
		$I_m\cap I_{m+1} \neq \emptyset$ for every $m \in (4/p^2,n/2-1)\cap \mathbb{N}$.
		This is equivalent to
		\begin{align}\label{eq:intervals}
			\left(\dfrac{2^{19}p^2n^2}{ m}\right)^{\frac{1}{2-\delta}} \le \frac{\sigma}{2^8p(m+1)^{1/2}}.
		\end{align}
		By simplicity, let $\gamma>0$ be such that $1/2 + \gamma = (2-\delta)^{-1}$. 
		As $m \ge 1$, it follows that
        $(m+1)^{1/2}/m^{1/2+\gamma} \le 2^{1/2}$.
        Moreover, as $\sigma \ge 2^{-4}n^4p^5$,~\eqref{eq:intervals} holds if $2^{22+19\gamma} \le (np)^{3-2\gamma}$,
        which is true for all $\gamma \in (0,1)$ and sufficiently large $n$.
        As $\delta \in(0,2/5)$ implies $\gamma \in (0,1/8)$ and $\delta \ge 2\gamma$, we conclude that 
		\begin{align*}
			|\Ex{e^{itX/\sigma }}|\le \exp  \left( -t^{2\gamma} \right)
		\end{align*}
		for every $t \in \big((2^{21}p^2 n)^{1/2+\gamma},\sigma/2^{10}\big)$.
	\end{proof}
 \vspace*{0.2cm}
	\section{The role of a single vertex in the regime $\sigma/2^{12} \leq t \leq \pi \sigma$}\label{section:end}
	\vspace*{0.2cm}
	Here we prove Theorem~\ref{thm:main-4}, which covers the last regime where $t$ has the same order of $\sigma$.
	Recall the definition of $\alpha_f$ and $\alpha$ given in~\eqref{eq:defalphaf} and~\eqref{eq:defalpha}, respectively.
	Here we still use the $\alpha$ function as a tool for bounding $|\Ex{e^{itX/\sigma}}|$, but we consider the case where $P_3$ consists of a single vertex. 
	As in Section~\ref{sec:alpha}, we start by proving an analogue of Lemma \ref{lemma:typical_alpha}.
	Recall that for an $m$-partition $P_1 \cup P_2 \cup P_3$ of $[n]$, we set $A = \{uv: u \in P_1, v \in P_2\}$.

	\begin{lemma}\label{lemma:Aprimeforonevertex}
		Let $n \in \mathbb{N}$, $P_1 \cup P_2 \cup P_3$ be an $1$-endowed partition of $[n]$ and $p=p(n) \in (0,1/2)$. 
		Let $A'\subset A$ be the set of pairs $f$ such that $|\alpha_f|=1$. We have $|A'|\geq(pn)^2/2^4$
		with probability at least $1-\exp(-pn/2^4)$.
	\end{lemma}
	
	\begin{proof}
		Let us set $P_3 = \{w\}$. For an edge $f=uv \in A$ we have
		\[\alpha_f = (x_{uw}^0-x_{uw}^1)(x_{vw}^0-x_{vw}^1).\]
		\noindent In words, $|\alpha_f|=1$ if and only if $u$ and $v$ are in the symmetric difference of $N_{G^0}(w)$ and $N_{G^1}(w)$, denoted by $N_{G^0}(w)\Delta N_{G^1}(w)$. Under the probabilities distributions of $G^0$ and $G^1$, we have that $|N_{G^0}(w)\Delta N_{G^1}(w)|$ has distribution $\text{Bin}(n-1, 2p(1-p))$. 
		Thus, by Chernoff's inequality (c.f. Lemma~\ref{lemma:chernoff}) we obtain
		\[\Pr{|N_{G^0}(w)\Delta N_{G^1}(w)| \leq p(1-p)n} \leq \exp(-p(1-p)n/2^3) \leq \exp(-pn/2^4).\]
		We conclude that the set $A'$ of pairs $f=uv$ such that $|\alpha_f| = 1$ has size at least
		\[|A'|\geq \binom{pn/2}{2} \geq \frac{(pn)^2}{2^4}\]
		with probability at least $1-\exp(-pn/2^4)$, as desired.
	\end{proof}
	
	Now we are ready to prove Theorem~\ref{thm:main-4}.
	
	\begin{proof}[Proof of Theorem \ref{thm:main-4}]
		The proof follows in the same fashion as in the proof of Theorem \ref{thm:main-3}. 
		Let $P_1 \cup P_2 \cup P_3$ be an $1$-partition of $[n]$.
		Let $A'\se A$ be the set of pairs $f$ such that $|\alpha_f|=1$.
		Recall that $A'$ only depends on $x_B^{0,1}$.
		By combining Lemma~\ref{lemma:charfuncbinomial}
		with  Lemma~\ref{lemma:CStrick-2}, we have
		\begin{align*}
			|\Ex{e^{itX/\sigma }}|^4 \le \cE_{x_B^{0,1}} |\cE_{x_{A}^0} ( e^{it\alpha/\sigma } )|
			&\le \cE_{x_B^{0,1}} |\cE_{x_{A'}^0} ( e^{it\alpha/\sigma } )|\\
			&\le 
			\cE_{x_B^{0,1}} \left (\prod \limits_{f \in A'} \left(1-8p(1-p)\left \| \dfrac{t\alpha_f}{2\pi \sigma} \right \|^2 \right) \right).
		\end{align*}
		As $t\leq \pi\sigma$ and $|\alpha_f|=1$ for $f\in A'$, we have $\| t\alpha_f/(2\pi \sigma)\| = t\alpha_f/(2\pi \sigma)$.
		As $p \le 1/2$, we obtain
		\begin{align}\label{eq:Aprimeforone}
				|\Ex{e^{itX/\sigma }}|^4 \le \cE_{x_B^{0,1}} \left ( \exp \left(-\dfrac{pt^2|A'|}{\pi^2\sigma^2}\right) \right ).
		\end{align}
	
	By Lemma~\ref{lemma:Aprimeforonevertex}, we have $|A'| \ge (pn)^2/2^4$ with probability at least $1-\exp(-pn/2^4)$.
	As $(\alpha +\beta)^{1/4}\le \alpha^{1/4}+\beta^{1/4}$ for every $\alpha, \beta \ge 0$ and $p \ge n^{-1/2}$,
	it follows from~\eqref{eq:Aprimeforone} that
	\begin{align*}
		|\Ex{e^{itX/\sigma }}| &\le \exp \left(-\dfrac{p^3n^2t^2}{2^{10}\sigma^2}\right)+\exp(-pn/2^6)\\
		& \le \exp\left(-\Omega(\sqrt{n}) \right),
	\end{align*}
	which finishes the proof.
	\end{proof}

	\section*{Acknowledgements}
	
	We would like to thank Rob Morris and Tibor Szab\'o for helpful suggestions which improved the presentation of this paper.
    We are also grateful to the referee for their careful reading and insightful comments, which have further refined the manuscript.

	\bibliographystyle{abbrv}
	
	\small{\bibliography{bib}}

\begin{thebibliography}{10}

\bibitem{barbour1982poisson}
A.~Barbour.
\newblock Poisson convergence and random graphs.
\newblock {\em Mathematical Proceedings of the Cambridge Philosophical Society}, 92(2):349--359, 1982.

\bibitem{barbour1989central}
A.~D. Barbour, M.~Karo{\'n}ski, and A.~Ruci{\'n}ski.
\newblock A central limit theorem for decomposable random variables with applications to random graphs.
\newblock {\em Journal of Combinatorial Theory, Series B}, 47(2):125--145, 1989.

\bibitem{berkowitz2016quantitative}
R.~Berkowitz.
\newblock A quantitative local limit theorem for triangles in random graphs.
\newblock {\em arXiv:1610.01281}, 2016.

\bibitem{berkowitz2018local}
R.~Berkowitz.
\newblock A local limit theorem for cliques in {$G(n,p)$}.
\newblock {\em arXiv:1811.03527}, 2018.

\bibitem{durrett}
R.~Durrett.
\newblock {\em Probability: Theory and Examples}.
\newblock Cambridge Series in Statistical and Probabilistic Mathematics, 2019.

\bibitem{erd6s1960evolution}
P.~Erd\H{o}s and A.~R{\'e}nyi.
\newblock On the evolution of random graphs.
\newblock {\em Publ. Math. Inst. Hungar. Acad. Sci}, 5:17--61, 1960.

\bibitem{feller}
W.~Feller.
\newblock {\em An introduction to probability theory and its applications, Volume 2}, volume~81.
\newblock John Wiley \& Sons, 1991.

\bibitem{10.1214/20-AOP1490}
J.~Fox, M.~Kwan, and L.~Sauermann.
\newblock {Anti-concentration for subgraph counts in random graphs}.
\newblock {\em The Annals of Probability}, 49(3):1515 -- 1553, 2021.

\bibitem{fox_kwan_sauermann_2021}
J.~Fox, M.~Kwan, and L.~Sauermann.
\newblock Combinatorial anticoncentration inequalities, with applications.
\newblock {\em Mathematical Proceedings of the Cambridge Philosophical Society}, 171(2):227--248, 2021.

\bibitem{frieze2016introduction}
A.~Frieze and M.~Karo{\'n}ski.
\newblock {\em Introduction to random graphs}.
\newblock Cambridge University Press, 2016.

\bibitem{gilmer2016local}
J.~Gilmer and S.~Kopparty.
\newblock A local central limit theorem for triangles in a random graph.
\newblock {\em Random Structures \& Algorithms}, 48(4):732--750, 2016.

\bibitem{janson1994orthogonal}
S.~Janson.
\newblock {\em Orthogonal decompositions and functional limit theorems for random graph statistics}, volume 534.
\newblock American Mathematical Soc., 1994.

\bibitem{janson2002infamous}
S.~Janson and A.~Ruci{\'n}ski.
\newblock The infamous upper tail.
\newblock {\em Random Structures \& Algorithms}, 20(3):317--342, 2002.

\bibitem{kar-ruci}
M.~Karo{\'{n}}ski and A.~Ruci{\'{n}}ski.
\newblock On the number of strictly balanced subgraphs of a random graph.
\newblock {\em Graph Theory}, pages 79--83, 1983.

\bibitem{kim2000concentration}
J.~H. Kim and V.~H. Vu.
\newblock Concentration of multivariate polynomials and its applications.
\newblock {\em Combinatorica}, 20(3):417--434, 2000.

\bibitem{nowicki1988subgraph}
K.~Nowicki and J.~C. Wierman.
\newblock Subgraph counts in random graphs using incomplete {U}-statistics methods.
\newblock {\em Discrete Mathematics}, 72(1-3):299--310, 1988.

\bibitem{paley1932note}
R.~E. Paley and A.~Zygmund.
\newblock A note on analytic functions in the unit circle.
\newblock {\em Mathematical Proceedings of the Cambridge Philosophical Society}, 28(3):266--272, 1932.

\bibitem{rollin2015local}
A.~R{\"o}llin and N.~Ross.
\newblock {Local limit theorems via Landau–Kolmogorov inequalities}.
\newblock {\em Bernoulli}, 21(2):851 -- 880, 2015.

\bibitem{pikachu}
N.~Ross.
\newblock {\em Fundamentals of Stein’s method}.
\newblock 2011.

\bibitem{rucinski1988small}
A.~Ruci{\'n}ski.
\newblock When are small subgraphs of a random graph normally distributed?
\newblock {\em Probability Theory and Related Fields}, 78(1):1--10, 1988.

\bibitem{sah2022local}
A.~Sah and M.~Sawhney.
\newblock Local limit theorems for subgraph counts.
\newblock {\em Journal of the London Mathematical Society}, 105(2):950--1011, 2022.

\end{thebibliography}

\end{document}